\documentclass[12pt]{amsart}
\usepackage{amsmath,amssymb,amsbsy,amsfonts,latexsym,amsopn,amstext,cite,
                                               amsxtra,euscript,amscd,bm,mathabx}
\usepackage{url}
\usepackage[colorlinks,linkcolor=blue,anchorcolor=blue,citecolor=blue,backref=page]{hyperref}
\usepackage{color}
\usepackage{graphics,epsfig}
\usepackage{graphicx}
\usepackage{float} 
\usepackage[english]{babel}
\usepackage{mathtools}
\usepackage{todonotes}
\usepackage{url}
\usepackage[colorlinks,linkcolor=blue,anchorcolor=blue,citecolor=blue,backref=page]{hyperref}
\usepackage[norefs,nocites]{refcheck}

\hypersetup{breaklinks=true}

\def \newell{r}
\def \newDelta{\Delta^*}
\def \newa{a^*}
\def \newb{b^*}

\usepackage[norefs,nocites]{refcheck}
\usepackage[english]{babel}
\begin{document}

\newtheorem{thm}{Theorem}
\newtheorem{lem}[thm]{Lemma}
\newtheorem{claim}[thm]{Claim}
\newtheorem{cor}[thm]{Corollary}
\newtheorem{prop}[thm]{Proposition} 
\newtheorem{definition}[thm]{Definition}
\newtheorem{rem}[thm]{Remark} 
\newtheorem{question}[thm]{Question}
\newtheorem{conj}[thm]{Conjecture}
\newtheorem{prob}{Problem}

\newtheorem{lemma}[thm]{Lemma}

\newcommand{\GL}{\operatorname{GL}}
\newcommand{\SL}{\operatorname{SL}}
\newcommand{\lcm}{\operatorname{lcm}}
\newcommand{\ord}{\operatorname{ord}}
\newcommand{\Op}{\operatorname{Op}}
\newcommand{\Tr}{\operatorname{Tr}}
\newcommand{\Nm}{\operatorname{Nm}}

\numberwithin{equation}{section}
\numberwithin{thm}{section}
\numberwithin{table}{section}

\def\vol {{\mathrm{vol\,}}}
\def\squareforqed{\hbox{\rlap{$\sqcap$}$\sqcup$}}
\def\qed{\ifmmode\squareforqed\else{\unskip\nobreak\hfil
\penalty50\hskip1em\null\nobreak\hfil\squareforqed
\parfillskip=0pt\finalhyphendemerits=0\endgraf}\fi}

\def \balpha{\bm{\alpha}}
\def \bbeta{\bm{\beta}}
\def \bgamma{\bm{\gamma}}
\def \blambda{\bm{\lambda}}
\def \bchi{\bm{\chi}}
\def \bphi{\bm{\varphi}}
\def \bpsi{\bm{\psi}}
\def \bomega{\bm{\omega}}
\def \btheta{\bm{\vartheta}}

\newcommand{\bfxi}{{\boldsymbol{\xi}}}
\newcommand{\bfrho}{{\boldsymbol{\rho}}}

\def\Kab{\sfK_\psi(a,b)}
\def\Kuv{\sfK_\psi(u,v)}
\def\SaUV{\cS_\psi(\balpha;\cU,\cV)}
\def\SaAV{\cS_\psi(\balpha;\cA,\cV)}

\def\SUV{\cS_\psi(\cU,\cV)}
\def\SAB{\cS_\psi(\cA,\cB)}

\def\Kmnp{\sfK_p(m,n)}

\def\KKap{\cH_p(a)}
\def\KKaq{\cH_q(a)}
\def\KKmnp{\cH_p(m,n)}
\def\KKmnq{\cH_q(m,n)}

\def\Klmnp{\sfK_p(\ell, m,n)}
\def\Klmnq{\sfK_q(\ell, m,n)}

\def \SALMNq {\cS_q(\balpha;\cL,\cI,\cJ)}
\def \SALMNp {\cS_p(\balpha;\cL,\cI,\cJ)}

\def \SACXMQX {\fS(\balpha,\bzeta, \bxi; M,Q,X)}

\def\SAMJp{\cS_p(\balpha;\cM,\cJ)}
\def\SAMJq{\cS_q(\balpha;\cM,\cJ)}
\def\SAqMJq{\cS_q(\balpha_q;\cM,\cJ)}
\def\SAJq{\cS_q(\balpha;\cJ)}
\def\SAqJq{\cS_q(\balpha_q;\cJ)}
\def\SAIJp{\cS_p(\balpha;\cI,\cJ)}
\def\SAIJq{\cS_q(\balpha;\cI,\cJ)}

\def\RIJp{\cR_p(\cI,\cJ)}
\def\RIJq{\cR_q(\cI,\cJ)}

\def\TWXJp{\cT_p(\bomega;\cX,\cJ)}
\def\TWXJq{\cT_q(\bomega;\cX,\cJ)}
\def\TWpXJp{\cT_p(\bomega_p;\cX,\cJ)}
\def\TWqXJq{\cT_q(\bomega_q;\cX,\cJ)}
\def\TWJq{\cT_q(\bomega;\cJ)}
\def\TWqJq{\cT_q(\bomega_q;\cJ)}

 \def \xbar{\overline x}
  \def \ybar{\overline y}

\def\cA{{\mathcal A}}
\def\cB{{\mathcal B}}
\def\cC{{\mathcal C}}
\def\cD{{\mathcal D}}
\def\cE{{\mathcal E}}
\def\cF{{\mathcal F}}
\def\cG{{\mathcal G}}
\def\cH{{\mathcal H}}
\def\cI{{\mathcal I}}
\def\cJ{{\mathcal J}}
\def\cK{{\mathcal K}}
\def\cL{{\mathcal L}}
\def\cM{{\mathcal M}}
\def\cN{{\mathcal N}}
\def\cO{{\mathcal O}}
\def\cP{{\mathcal P}}
\def\cQ{{\mathcal Q}}
\def\cR{{\mathcal R}}
\def\cS{{\mathcal S}}
\def\cT{{\mathcal T}}
\def\cU{{\mathcal U}}
\def\cV{{\mathcal V}}
\def\cW{{\mathcal W}}
\def\cX{{\mathcal X}}
\def\cY{{\mathcal Y}}
\def\cZ{{\mathcal Z}}
\def\Ker{{\mathrm{Ker}}}

\def\NmQR{N(m;Q,R)}
\def\VmQR{\cV(m;Q,R)}

\def\Xm{\cX_m}

\def \A {{\mathbb A}}
\def \B {{\mathbb A}}
\def \C {{\mathbb C}}
\def \F {{\mathbb F}}
\def \G {{\mathbb G}}
\def \L {{\mathbb L}}
\def \K {{\mathbb K}}
\def \PP {{\mathbb P}}
\def \Q {{\mathbb Q}}
\def \R {{\mathbb R}}
\def \Z {{\mathbb Z}}
\def \fS{\mathfrak S}
\def \fE{\mathfrak E}
\def \fA{\mathfrak A}

\def\GL{\operatorname{GL}}
\def\SL{\operatorname{SL}}
\def\PGL{\operatorname{PGL}}
\def\PSL{\operatorname{PSL}}
\def\li{\operatorname{li}}
\def\sym{\operatorname{sym}}

\def\Mob{M{\"o}bius }

\def\fF{\mathfrak{F}}
\def\M{\mathsf {M}}
\def\T{\mathsf {T}}

\def\e{{\mathbf{\,e}}}
\def\ep{{\mathbf{\,e}}_p}
\def\eq{{\mathbf{\,e}}_q}

\def\\{\cr}
\def\({\left(}
\def\){\right)}
\def\fl#1{\left\lfloor#1\right\rfloor}
\def\rf#1{\left\lceil#1\right\rceil}

\def\Tr{{\mathrm{Tr}}}
\def\Nm{{\mathrm{Nm}}}
\def\Im{{\mathrm{Im}}}

\def \oF {\overline \F}

\newcommand{\pfrac}[2]{{\left(\frac{#1}{#2}\right)}}

\def \Prob{{\mathrm {}}}
\def\e{\mathbf{e}}
\def\ep{{\mathbf{\,e}}_p}
\def\epp{{\mathbf{\,e}}_{p^2}}
\def\em{{\mathbf{\,e}}_m}

\def\Res{\mathrm{Res}}
\def\Orb{\mathrm{Orb}}

\def\vec#1{\mathbf{#1}}
\def \va{\vec{a}}
\def \vb{\vec{b}}
\def \vs{\vec{s}}
\def \vu{\vec{u}}
\def \vv{\vec{v}}
\def \vz{\vec{z}}
\def\flp#1{{\left\langle#1\right\rangle}_p}
\def\T {\mathsf {T}}

\def\sfG {\mathsf {G}}
\def\sfK {\mathsf {K}}

\def\mand{\qquad\mbox{and}\qquad}

 \date{\today}


\title{On character sums with determinants}

\author[\'E. Fouvry] {\'Etienne Fouvry}
\address{D{\'e}partement de Math{\'e}matiques, Universit{\'e} Paris-Saclay,
91405 Orsay Cedex, 
France}
\email{etienne.fouvry@universite-paris-saclay.fr}

\author[I. E. Shparlinski] {Igor E. Shparlinski}
\address{School of Mathematics and Statistics, University of New South Wales, Sydney NSW 2052, Australia}
\email{igor.shparlinski@unsw.edu.au}

\begin{abstract} We estimate weighted character sums with determinants $ad-bc $ of $2\times 2$ matrices
modulo a prime $p$ with entries $a,b,c,d $ varying over the interval $ [1,N]$. Our goal is to obtain 
non-trivial bounds for values of $N$ as small as possible. In particular, we achieve 
this goal, with a power  saving, for $N \ge  p^{1/8+\varepsilon}$ with any fixed $\varepsilon>0$, 
which is very likely to be  the best possible unless the celebrated Burgess bound  is improved.  By other techniques,  we also treat more general sums
but sometimes for larger values of $N$. 
 \end{abstract}


\maketitle

\hfill
{\it On the 50th anniversary of Chen's theorem\/}

\tableofcontents

\section{Introduction and statement of the results} 
\subsection{The original question} In the spring of 2022, S. Ganguly turned the attention of the first named author to the question of obtaining cancellations
in the quadruple sum of Legendre symbols
\begin{equation}\label{Gangsum} \sum_{\vert a \vert \leq x}\  \sum_{\vert b \vert \leq x} \ \sum_{\vert c \vert \leq x}\  \sum_{\vert d \vert \leq x}
\left( \frac{ad-bc}p \right),
\end{equation}
 where $p$ is a large prime, $x$ is as small as possible,  with the view  of counting the analogues of quadratic non-residues in the context of matrices over the finite field  $\mathbb F_p$ of $p$ elements.

 For a prime  $p$ and an integer  $N\geq 2$,  
we  identify $\mathbb F_p$ with the set of integers $\{0, \,1,  \ldots,  p-1\}$
and introduce the following notations,     
\begin{itemize}
\item $\mathcal M_2 (p) $ is the set of $2\times 2$--matrices, with coefficients in $\mathbb F_p$,
\item $ \mathcal M_2^* (p)\, (=  {\rm GL}(2, \mathbb F_p))$ is the subset of $\mathcal M_2 (p)$ containing all the non--singular matrices, 
\item $\mathcal M_2^{\ne \square}(p)$    is the subset of $\mathcal M_2^* (p)$ defined by 
$$
\mathcal M_2^{\ne \square}(p) := 
\left\{A \in \mathcal M_2^* (p) :~A\ne B^2 \text{ for all } B \in \mathcal M_2^* (p) \right\}, 
$$ 
the elements of  $\mathcal M_2^{\ne  \square}(p)$ are called {\it matrices without square root.}
\item ${\mathbf M} (N,p)$ is the subset of $\mathcal M_2 (p)$ defined by 
$$
{\mathbf M }(N,p) : 
=\left\{ 
 \begin{pmatrix}  a \bmod p & b \bmod p\\ c \bmod p & d \bmod p
 \end{pmatrix}:~0\leq a,\, b,\, c,\, d \leq N 
\right\}.
$$
\end{itemize}  

The question now is:

\begin{question}\label{kappa2*}
Estimate the size of the set 
$$
{\mathbf  M } (N,p) \cap \mathcal M_2^{\ne \square}(p),
$$
and in particular,  find a real number $0 <\kappa_2 <1$,    as small as possible,  such that for   sufficiently large $p$,  one has  
\begin{equation}\label{kappa2}
{\mathbf M }(p^{\kappa_2},p)\, \bigcap\, \mathcal M_2^{\ne \square}(p) \ne \emptyset.
\end{equation}
\end{question} 

 The restriction of this question to invertible matrices of $\mathcal M_2 ^* (p)$ comes from the observation that for every $b \bmod p$ with $b \ne 0$, there exists no matrix $B \in \mathcal M_2 (p)$
 such that
 $$
 B^2 =\begin{pmatrix}
 0 & b\\
 0 & 0
 \end{pmatrix}.
 $$
In other words, it is easy to find a matrix   $A \in \mathcal M_2 (p)$ which only has small entries and which is not a square. However, the above method produces a singular matrix. 

Recall now the situation in the context of matrices with dimension one, which means $\mathbb F_p$. Let $\mathfrak z_p$ be    the least integer  $n \geq 2$ which is not a quadratic residue modulo $p$. 
A well-known question in number theory is to give a bound for $\mathfrak z_p$, in the following sense:

\begin{quotation}
{\it Find a real number $0<\kappa_1 <1$, as small as possible,   such that, for sufficiently
large $p$, one has 
\begin{equation}\label{kappa1}
2\leq \mathfrak z_p\leq p^{\kappa_1}. 
\end{equation}
}
\end{quotation}  

It is universally believed  that~\eqref{kappa1} is satisfied by any $0<\kappa_1 <1$.   This conjecture is considered as very difficult. 
 It is time to recall that we know that~\eqref{kappa1} holds for any 
 \begin{equation}\label{4/sqrte}
 \kappa_1 > 1/(4 \sqrt \mathrm {e}).
 \end{equation}  This result 
 is the content of~\cite[Theorem~2]{Buresiduesnonresidues}. It  is a consequence of  the famous Burgess 
 bound~\cite{Buresiduesnonresidues} for  short sums of characters that we recall in Lemma~\ref{Burgessclassical} below, together 
 with a sieving idea of the work of Vinogradov~\cite{Vi}.

 A  first  way to answer  to  Question~\ref{kappa2*} is to use the properties of the determinant and of the Legendre symbol to deduce the inclusion 
 \begin{equation}\label{inclusion}
\left\{ A \in \mathcal M_2^* (p) :~\left(\frac{\det A}p \right) =-1 \right\}  \subseteq \mathcal M_2^{\ne \square}(p),
\end{equation}
but we lose information by this inclusion (see the  \ref{app:A2} below). We are led to study the least element of the left hand side of~\eqref{inclusion}.
If $\kappa_1$ satisfies~\eqref{kappa1} then any  $\kappa_2 > \kappa_1/2$ satisfies~\eqref{kappa2}; this implication  is elementary but we present it here for completeness.

\begin{thm} 
Let $0<\kappa_1<1$ and $C$ such that, for all $p\geq 3$, one has the inequality
\begin{equation}\label{zp<}
2\leq \mathfrak z_p \leq C p^{\kappa_1}.
\end{equation}
Then, for all $p\geq 3$ one has 
\begin{equation}\label{MinterM}
{\mathbf M }\left(\sqrt C p^{\kappa_1/2} +1 ,p\right)\,  \bigcap\, \mathcal M_2^{\ne \square}(p) \ne \emptyset.
\end{equation}
\end{thm}

\begin{proof} 
 We  want to express the  least  quadratic non-residue $\mathfrak z_p$ as
 $$
 \mathfrak z_p= ad-bc,
 $$
 with  positive $a$, $b$, $c$ and $d$ of  suitable sizes. So we choose 
 $$
  a= \rf{\mathfrak z_p^{1/2}}\, (\leq \mathfrak z_p^{1/2} +1), \quad b\equiv -\mathfrak z_p \bmod a \quad \text{with  }1  \leq  b \leq a \text{ and }
 c=1.
 $$
 By construction, the number 
$$
 d = \frac {\mathfrak z_p +bc }{a}  
$$
  is an integer  and satisfies the inequalities
 $$
 0  < d \leq \frac{\mathfrak z_p + \rf{\mathfrak z_p^{1/2}}}{\rf{\mathfrak z_p^{1/2}}} \leq \mathfrak z_p^{1/2}+1.
 $$
 Then, by~\eqref{zp<},  we have the inequalities
 $$
 1 \leq a,\, b,\, c,\, d \leq  \sqrt C p^{\kappa_1/2}+ 1
 $$
 and the result follows.
 \end{proof}

  So, in view of~\eqref{4/sqrte},  we have proved  that~\eqref{kappa2} is satisfied with any 
  $$
  \kappa_2 > 1/(8 \sqrt \mathrm {e}).
  $$ 
  
Our purpose now is to use more sophisticated tools to describe the cardinality of the set appearing on the left hand side of~\eqref{MinterM}. The first of these tools
is a lower bound for the number of quadratic non-residues in some interval beginning at $\mathfrak z_p$ (see Lemma~\ref{599}). The second one concerns with the number of
solutions to the determinant equation $\Delta = ad-bc$ (see Lemma~\ref{BettinChandee}). 
We have

\begin{thm} \label{413}
For every  $\varepsilon >0$ there exists  some $C> 0$ and $p_0$, which depend only on $\varepsilon$, such that for every $p\geq p_0$,   for every $N\geq p^{1/8\sqrt \mathrm {e }+\varepsilon}$ one has the inequality
$$
\sharp \, \left({\mathbf{M}} (N,p)\, \bigcap \, \mathcal M_2^{\ne \square} (p)\right) \geq C N^4.
$$
\end{thm}  

Of course,  for every $p$ and every  $1\leq N <p$ one has the inequality 
$$ \sharp \, \left({\mathbf M} (N,p) \, \bigcap \,   \mathcal M_2^{\ne \square} (p)\right)  \leq \sharp {\mathbf M}(N, p) \leq (N+1)^4,
$$
so  Theorem~\ref{413} gives the correct order of magnitude. We  give its proof in
 \S\ref{proofTheorem13} and it is obvious  that its counting process can be generalized 
in several directions. For instance, applying Lemma~\ref{BettinChandee} with  $\alpha_{n_1}$ and $\beta_{n_2}$  to be the characteristic functions of the set of primes,  our argument yields the lower bound
$$
\sharp \, \left({\mathbf M}^\dag (N,p) \, \bigcap \,   \mathcal M_2^{\ne \square} (p)\right) \geq C N^4 (\log p)^{-2},
$$
where ${\mathbf M}^\dag (N,p)$ is the subset  of ${\mathbf M} (N,p)$ obtained by restricting the entries of the second row of the matrix  to prime values.  

The  above questions and results make it natural to ask about the size of  
$\mathcal M_2^{\ne \square}(p)$. Quite surprisingly, as far as we know, this questions has never been addressed in the literature. In \ref{app:A2} we  present some elementary arguments which give the asymptotic formula 
 \begin{equation}
 \label{eq: SquareCount} 
\sharp \mathcal M_2^{\ne \square}(p) = \frac{5}{8} p^4 + O(p^3),
\end{equation}
and in fact with somewhat tedious dealing with several discarded cases (absorbed in the error term $O(p^3)$), 
one can easily derive an explicit closed form formula for $\sharp \mathcal M_2^{\ne \square}(p)$.

Finally, one may also consider the $\GL(2, \mathbb{F}_p)$ analogue of primitive roots modulo a prime and ask questions similar to the ones considered above. A study of such questions has been initiated in~\cite{GaRa}. 
See, in particular,~\cite[\S1.5]{GaRa}. 
   \subsection{Oscillation of characters} When studying the least element of the left hand side of~\eqref{inclusion}, 
 we are naturally led  to search for cancellations in the sum  of Legendre symbols 
 introduced in~\eqref{Gangsum}
and more generally in the sum
$$
S(N, \chi) := 
   \sum_{1\leq a\leq N} \sum_{1\leq b \leq N} \sum_{1\leq c \leq N} \sum_{1\leq d \leq N}  \chi (ad-bc),
$$
   where $\chi$ is a non-principal multiplicative character modulo $p$, 
   we refer to~\cite[Section~3]{IwKow} for a background on characters.  
      
   Before stating the other results, we recall the following convention: As usual, the notations  $F \ll G$ and $F = O(G)$, are equivalent
to $|F|  \le c G$ for some constant $c>0$,
which throughout the paper may depend on the real $\varepsilon>0$ (with or without subscripts).  
   
   The results that we  present can be extended to the larger set of summation $0\leq a,\, b,\, c,\, d\leq N$, since the contribution of the terms with $abcd =0$ is  easily treated. 
   Keeping the variables $a$, $b$ and $c$ fixed, and applying the Burgess bound~\cite{Buresiduesnonresidues} (see 
   Lemma~\ref{Burgessclassical} below) to the sum in $d$, we deduce that, for every $\varepsilon >0$ there exists $\delta$, such that, for $N\geq p^{\frac 14 +\varepsilon}$ one has the bound $S(N, \chi) \ll N^{4-\delta}$.
 By this trivial  approach, we did not benefit from cancellations coming from the other variables. Our next purpose is   to efficiently take advantage of the other three variables.

  Actually we do not benefit from the smoothness of each of these variables, some of them can be attached with general coefficients. We also work with the sum of the modulus
  of some interior sums. Clearly, all these cases contain $S(N, \chi)$ as a particular case.
  
  The first sum that we study appealing to the properties of the determinant equation  is
  $$
  U_\chi (\boldsymbol \alpha, \boldsymbol \beta, N ) := \sum_{1\leq a\leq N}\alpha_a \sum_{1\leq b \leq N} \beta_b \sum_{1\leq c \leq N} \sum_{1\leq d \leq N} \chi (ad-bc),
  $$ 
  where $\boldsymbol \alpha =(\alpha_a)$ and $\boldsymbol \beta= (\beta_b)$ are quite general sequences.

   \begin{thm}\label{thesumU} For every  $\varepsilon >0$ there exist $\delta_0 >0$ such that the following inequality holds
   $$
   U_\chi  (\boldsymbol \alpha, \boldsymbol \beta, N ) \ll  N^{4-\delta_0},
   $$
   uniformly for bounded weights $\boldsymbol \alpha$ and $\boldsymbol \beta$, for all $a$ and $b \geq 1$
   and uniformly for $p^{1/2} > N\geq p^{1/8 + \varepsilon}$.
   \end{thm}

   Since, with the above conditions, we have $\vert ad -bc \vert \leq N^2 $, the exponent $1/8$ somehow is optimal by comparison with the critical exponent $1/4$ appearing in Burgess bound mentioned in Lemma~\ref{Burgessclassical} below.
   
The second sum that we  study is 
   $$  T_\chi( A,B,C, \cD;\balpha) : =    \sum_{1 \le a \le A} \sum_{1\le b \le  B} \sum_{1\le c\le C} 
 \left| \,\sum_{d \in \cD} \alpha_d \chi(ab-cd)\, \right|, $$ 
   where $A$, $B$, $C$ and $D$ are integers $ \geq 1$, $\cD\subseteq \mathbb F_p^*$ has cardinality $D$, and where $\chi$ is a non-principal character modulo $p$.   As we have 
   mentioned, we identify $\mathbb F_p^*$ with the set of integers $\{ 1, \,2,  \ldots, p-1\}$.

  For this sum, where no variable of summation is smooth,  we  prove the following result.
  
  \begin{thm}\label{thm:SumTD1} For every integer  $\nu \geq 1$, there exists a constant $C(\nu)$ such that,  for all prime $p$, for all non-principal character $\chi$ modulo $p$,
for all   positive integers $A$, $B$, $C$ and $D$ satisfying 
$$ABC\text{ and } D < p, 
$$
 for all set $\cD \subseteq \F_p^*$ 
of cardinality $\sharp\cD = D$, 
 and for all  arbitrary complex weights $\balpha$ satisfying 
 \begin{equation}\label{leq1}
 \vert \alpha_d\vert \leq 1 \text{ if } d\in \mathcal D,
 \end{equation}
  we have the inequality 
\begin{align*}
T_\chi(& A,B,C, \cD;\balpha)  \\
& \leq C(\nu) ABCD \( \(\frac{p}{ABC}\)^{1/(2\nu)} D^{-1/2} +  \(\frac{p^{1/2}}{ABC}\)^{1/(2\nu)}\)
 (\log p)^{4/\nu}. 
\end{align*}
 
\end{thm}

Returning to Theorem~\ref{thm:SumTD1} and choosing $\nu$ sufficiently large (depending on $\varepsilon$), we deduce  the following upper bound. 

\begin{cor}\label{cor:SumTD1}
There exists a constant $c >0$, such that, for every  $\varepsilon>0$, for every prime $p$,
for  every non-principal character $\chi$, for every  positive integers $A$, $B$, $C$  and $D$ with  
$$
p > ABC \ge  p^{1/2 + \varepsilon} \text{ and } D\ge p^{\varepsilon},
$$
for every  $\cD \subseteq \F_p^*$ 
of cardinality $\sharp\cD = D$, for every 
  complex weights $\balpha$ satisfying~\eqref{leq1}  we have the inequality
$$
T_\chi( A,B,C, \cD;\balpha)  \ll   ABCD  p^{- c \varepsilon^2}.
$$
\end{cor}

Corollary~\ref{cor:SumTD1} immediately leads to a bound for the original sum $S(N, \chi)$, by choosing $A=B=C=D= N$, but the obtained bound is interesting for $N >p^{1/6+ \varepsilon}$
which is worse than what we obtain by Theorem~\ref{thesumU}.

  The last sum we are concerned with  is   
$$
  T_\chi( N) :=   \sum_{1 \le a \le N} \sum_{1\le b \le  N} \sum_{1\le c\le N}   \left| \sum_{1\le d\le N}  \chi(ab-cd)\right|,
$$
  where $N\geq 1$. This sum is more general than  $U_\chi (\boldsymbol \alpha, \boldsymbol \beta, N )$. It is treated by different methods to obtain    the following upper bound.
 
   \begin{thm}\label{thm:SumTN1}
For any    $\varepsilon>0$, one has the inequality 
$$
T_\chi( N)  \ll   N^4   \frac{\log \log p}{\log p},
$$ for every prime $p$,
for every non-principal character $\chi$ modulo $p$, and for every  
\begin{equation}\label{N>p1/8}
p >N  \ge  p^{1/8 + \varepsilon}.
\end{equation} 
\end{thm} 

 It is worth noticing that the condition~\eqref{N>p1/8} co\" incides with the condition appearing in Theorem~\ref{thesumU}.

\begin{rem} We note that our use of absolute values in the above character sums 
is equivalent to attaching weights (to the variables outside of these absolute values)
which are bounded in the $L_\infty$-norm.
Simple modifications of our arguments also allow to introduce weights for which we control 
$L_1$- and $L_2$-norms. \end{rem} 
  \section{Conventions and classical tools} 
   \subsection{Conventions} 
   \label{sec:conv} 
 The additive character with period $1$ is denoted by
$$
\xi \in \mathbb R \mapsto \e(\xi) =  \exp\(2\pi \mathrm{ i} \xi\).
$$

For a finite set $\cS$ we use $\sharp\cS$ to denote its cardinality. 

In summation ranges we sometimes write the Legendre symbol modulo of some bulky expression $\fF$  as $(\fF/p)$.   

 For $Y \geq 1$, the characteristic function of the interval $(Y/2, Y]$ is denoted by $\mathbf 1_Y$.

The letter $p$ is reserved for prime numbers.

   \subsection{Around Burgess bound}
   
The Burgess bound~\cite{Buresiduesnonresidues} concerns short sums of values of characters over consecutive  integers. A typical modern version of this bound is of the following form (see, for instance,~\cite[Theorem~12.6 and Equation~(12.58)]{IwKow}).

   \begin{lemma} \label{Burgesslog} For every $r\geq 2$, there exists two  constants $c(r)$ and $C(r)$, such that for every prime $p$, for every non-principal character $\chi \bmod p$, for every $M$ and $ N \geq 1$,  one has the inequality
$$
\left\vert  \sum_{M\leq n \leq M+N} \chi (n) \right\vert \leq C(r) N^{1-1/r} p^{(r+1)/4r^2} (\log p)^{c(r)}.
$$
   \end{lemma}  
      
 The value of the constant $c(r)$ has been the subject of improvements (see~\cite{dlBMT,KSY}). However,  the improvement of the exponent $1/4$ in the following statement remains 
 a famous challenging problem.
 
   Choosing $r$ in an optimal way, we obtain from Lemma~\ref{Burgesslog} the following version, which is well suited for applications:
   \begin{lemma}\label{Burgessclassical} For every $\varepsilon >0$, there exists $\delta>0$ (depending on $\varepsilon$), such that for every prime $p$, for every  
   non-principal character $\chi $ modulo $p$,  for every $M$ and  for every $N \geq p^{1/4 + \varepsilon}$  one has the inequality 
   $$
   \sum_{M\leq n \leq M+N} \chi (n) \ll N p^{-\delta}.
   $$
   \end{lemma}
   The important tool in Burgess proof is the Riemann Hypothesis for curves proved by Weil~\cite{We1,We2}.  It leads to the following important step, that we call the {\it Davenport--Erd\"os Lemma} to refer  to the seminal result in~\cite[Lemma~3 and footnote p.~262]{DaEr}.

 \begin{lemma}
\label{lem: D-E Lem*1} For any prime $p$, for every non-principal multiplicative character $\chi$ modulo $p$, 
for any set $\cD \subseteq \F_p^*$ with  $\sharp\cD = D$, for any complex 
weights $\balpha$ satisfying $\vert \alpha_d \vert \leq 1$ for $d \in \cD$ and for every integer  $\nu \geq 1$,  we have the inequality
$$
  \sum_{\lambda=1}^{p-1}  
   \left|\sum_{d \in \cD}\alpha_d \chi\(\lambda+d\)\right|^{2\nu} \le 
    \(2\nu D\)^\nu p + 2\nu D^{2\nu} p^{1/2}.
  $$
\end{lemma}

One easily checks  that the proof given in~\cite[p.~329]{IwKow}, which concerns the case $\cD =\{1,\dots, D\}$ and $\alpha_d=1$, 
actually   extends to arbitrary sets $\cD \subseteq \F_p^*$  and any bounded weights $\boldsymbol \alpha =(\alpha_d)_{d \in \mathcal  D}$ without any changes. 
\subsection{Counting quadratic non-residues with small size} In~\eqref{4/sqrte} we recalled the size of $\mathfrak z_p$, the least quadratic non-residue modulo $p$.
The following result gives a lower bound for the cardinality of the set of quadratic non-residues which satisfy this inequality

\begin{lemma} \label{599} For every $\varepsilon >0$, there exists $\delta>0$ and $p_0 $ such that, for every $p >p_0$,
one has the inequality
$$
\sharp \left\{ n: 1\leq n \leq p^{1/(4 \sqrt \mathrm {e}) +\varepsilon}, \, (n/p)=-1\right\}
\geq \delta p^{1/(4 \sqrt \mathrm {e}) +\varepsilon}.
$$
\end{lemma}

This is the content of~\cite[Theorem~2.1]{BGH-BS}. For an explicit value of $\delta$ see~\cite[Corollary~1.8]{GS}. 

\subsection{Bounds on some arithmetic  functions} For  an integer $s\geq 1$, let $\tau_s(m)$ denotes the $s$-fold divisor function, that is, 
 the number of ordered factorisations $m = k_1\ldots k_s$ with non-negative 
 integers $ k_1,\ldots, k_s$. 
 
 In particular $\tau_2(m) = \tau (m)$, the usual divisor function. On average the function $\tau_s (m)$ is bounded by a power of $\log m$ since we have
 the following esimate (see~\cite[Equation~(1.80)]{IwKow}, for instance)

 \begin{lemma}
\label{lem:tau31} Let $s\geq 1$ be a fixed integer. Then for every $M\geq 1$, one has 
the inequality
$$
  \sum_{1\leq m \le M} \tau_s(m)^2\ll M(\log 2M)^{s^2-1}.
  $$
  
\end{lemma}
 Finally, the following estimate is an easy consequence of the Prime Number Theorem
 (or in fact of Chebychev's inequality).

 \begin{lemma}\label{Cheby}
 Uniformly for $x\geq 1$, one has the inequality
 $$
 \sum_{p\geq x} \frac 1{p^2} \ll \frac {1}{x \log 2x}.
 $$
 \end{lemma}

\subsection{Sifted integers} Given two real numbers $y\ge x> 2$ and an integer $N \ge 1$,  we denote
by $\cA_0(N;x,y)$ the set  of positive integers $n \le N$ that do not have
a prime divisor in the  half-open interval $(x,y]$. We need the following upper bound
on the cardinality $\sharp\cA_0(N;x,y)$:

\begin{lemma}
\label{lem:ANxy1} Uniformly over  integers $N$ and real $x$ and $y$ with  $N \ge y \ge x \ge 2$,  we have
$$
\sharp\cA_0(N;x,y) \ll N \frac{\log x}{\log y}.
$$  
\end{lemma}

\begin{proof} A classical method from the sieve theory  leads to the upper bound 
$$
\sharp  \left\{ 1 < m \leq M  :~p\mid m \Rightarrow p \geq y\right\} = O \left( \frac M {\log 2y} \right), 
$$
uniformly for $M$ and $y \geq 1$. We write every $n\in \mathcal A _0(N;x,y)$  under the unique form
$$
n= dm,
$$
where $d$ has all its prime factors less than $x$
 and $m$ has all its prime factors greater than $y$. 
 With Lemma~\ref{lem:ANxy1}, we obtain
 the inequality
 \begin{align*}
 \sharp\cA_0(N;x,y)  &  \ll  1+ \sum_{\substack{d \\ p\mid d \Rightarrow p \leq x}} \frac{N/d} {\log 2y}\ll 1 + \frac N{\log 2y}  \prod_{p \leq x} \left( 1 +\frac 1p \right) \\
 & \ll N \frac{\log 2x}{\log 2y},
 \end{align*}
 by the Mertens formula (see, for example,~\cite[Equation~ 2.15]{IwKow}).  
 \end{proof}
 
 \subsection{The determinant equation} Due to the structure of the problem, we need some information on the so-called {\it determinant equation}
 $$
 \Delta = ad-bc,
 $$
 where $\Delta $ is a given integer, and where the unknowns $a$, $b$, $c$ and $d$ lie in some interval. The following result is an improvement of~\cite[Theorem~1]{DFI},
 which is based on bounds of sums of Kloosterman fractions. However, we remark that the 
 original result from~ \cite{DFI} is also sufficient for our purposes.

 We now recall  the following asymptotic formula due to Bettin and Chandee~\cite[Corollary~1]{BeCh}.

 \begin{lemma}\label{BettinChandee} Let $\Delta \ne 0$ be an integer and let
$$
 \mathcal T_\Delta (M_1,M_2,N_1, N_2):=
 \underset{\substack{m_1 \in \mathcal M_1,\, m_2 \in \mathcal M_2,\, n_1 \in \mathcal N_1,\, n_2 \in \mathcal N_2\\\ m_1n_2-m_2n_1 =\Delta}}{\sum\  \sum \ \sum \ \sum} f(m_1) g(m_2) \alpha_{n_1} \beta_{n_2},
$$
 where the functions $f(m_1)$, $g(m_2)$ and the weights $\balpha=\{\alpha_{n_1}\}$ and 
 $\bbeta = \{\beta_{n_2}\}$ are supported on 
 $ \mathcal M_1 := [M_1/2, M_1]$, $\mathcal M_2 := [M_2/2, M_2]$, $\mathcal N_1:= [N_1/2, N_1]$ and $\mathcal N_2 := [N_2/2, N_2]$, respectively.  Moreover assume
 that the functions $f$ and $g$ are of $\mathcal C^\infty$--class and satisfy the inequalities 
 $$
 f^{(j)}(t) \ll \eta^j M_1^{-j}, \quad t \in  \mathcal M_1, \mand g^{(j)}(t)  \ll \eta^{j} M_2^{-j},
 \quad t \in  \mathcal M_2,
 $$ 
  for all fixed $j\geq 0$  and some $\eta >1$.
 Then, for any $\varepsilon >0$,  we have
 \begin{equation} 
 \begin{split} 
 \label{formulaforTDelta}
 \mathcal T_\Delta (M_1,M_2,N_1,N_2) 
 & = \underset{\substack{n_1 \in \mathcal N_1,\, n_2 \in \mathcal N_2 \\ \gcd(n_1,n_2) \mid \Delta} }{ \sum \ \sum} 
 \frac{\gcd(n_1,n_2)}{n_1n_2} \alpha_{n_1} \beta_{n_2}\,
 \int_{\mathbb R} f \left( \frac{x+\Delta}{n_2} \right) g \left( \frac { x}{n_1} \right)\, dx\\
 & \qquad   +O \left( (\eta R)^{\frac{3}{2}}\Vert \balpha \Vert_2 \, \Vert \bbeta\Vert_2
  (N_1N_2)^{\frac{7}{20} }(N_1+N_2)^{\frac 14 +\varepsilon} (M_1M_2)^\varepsilon \right),
 \end{split}
 \end{equation} 
 where $\Vert \balpha \Vert_2$ and $\Vert \bbeta\Vert_2$ are $L_2$-norms of the weights
  $\balpha$ and $\bbeta$, respectively, and
 \begin{equation}\label{defR}
 R:= \frac{M_1N_2}{M_2N_1} + \frac{M_2N_1}{M_1N_2}.
 \end{equation}  
 \end{lemma}
 In our application, we  have $M_1$, $M_2$, $N_1$ and $N_2$  are approximately of the same size $X$, say,  so $R $ is of moderate size. Hence, if the size of  $\Delta$ 
 is about $M_1N_2$, and if  $\alpha_{n_1}$ and $\beta_{n_2} $ are the characteristic functions of the intervals $\mathcal N_1$ and $\mathcal N_2$, the main term of the above formula is approximately $X^2$ and the error term has size $O ( \eta^{\frac 32 } X^{\frac{39}{20} + \varepsilon})$.
 
 In the case where $\beta_{n_2} =1$ we can obtain an upper bound of quality 
 comparable with~\eqref{formulaforTDelta} but with simpler tools. Indeed, the deteminant equation $m_1n_2-m_2n_1= \Delta$
 is transformed into the congruence $m_1\equiv \Delta \overline{n_2} \bmod n_1$. It suffices to apply  classical Fourier techniques and bounds for short Kloosterman sums to obtain the desired result.
 \subsection{A smooth partition of unity} 
 The use of Lemma~\ref{BettinChandee} requires smooth functions. 
 To fulfil this condition, we appeal to the following smooth partitioning of unity (see, for instance~\cite[Lemme~2]{Fo}).

 \begin{lemma}\label{smoothpartition} Let $\Xi>1$ be a real number. There exists a sequence $b_{\ell, \Xi}$ ($\ell \geq 0$) of functions belonging to $\mathcal C ^\infty([\Xi^{\ell -1}, \Xi^{\ell +1}],
 \mathbb R^+)$, with the two properties
 \begin{itemize}
 \item for all $\xi \geq 1$, one has
 $$
 \sum_{\ell =  0}^\infty b_{\ell, \Xi} (\xi) =1,
 $$
 \item  for all $\xi \geq 1$, for all fixed $\nu\geq 0$, the derivatives satisfy 
 $$ b_{\ell, \Xi}^{(\nu)} (\xi ) \ll  \xi^{-\nu} \Xi^{\nu} (\Xi-1)^{-\nu}
 $$ 
 \end{itemize}
 \end{lemma}

 Notice that the above conditions imply the equality   $b_{\ell, \Xi} (\Xi^\ell) =1$ and the inequality $0\leq b_{\ell, \Xi} (t) \leq 1$ for all $t \geq 1$.  
 
 \section{Proof of Theorem~\ref{413}}\label{proofTheorem13} Assume that $N$   satisfies the inequalities $p^{1/(8 \sqrt {\mathrm e}) +\varepsilon} < N < \sqrt p$ (as otherwise the result follows easily 
 from classical tools).

 We start from the inequality
 \begin{equation} 
 \begin{split} 
 \sharp \, \left({\mathbf M} (N,p) \, \bigcap \,   \mathcal M_2^{\ne \square} (p)\right) &\geq \underset{\substack{1\leq a,\, b,\, c,\, d\leq N\\ ((ad-bc)/p)= -1}} {\sum\, \sum\, \sum\, \sum}\,  1  \\
 & \geq  \sum_{\substack{ 1-N^2 \leq \Delta  \leq N^2-1
 \\ (\Delta/p) =-1}} \left( \underset{\substack{ 1\leq m_1,\, m_2, \, n_1, \, n_2\leq N\\ m_1n_2-m_2 n_1 =\Delta}} {\sum\, \sum\, \sum\, \sum}\,  1  \right) \\
 &\geq  \sum_{\substack{1< \Delta  \leq N^2/1000
 \\ (\Delta/p) =-1}}   \mathcal T_\Delta(N,N,N,N),
  \label{721}
 \end{split}
 \end{equation} 
 where $ \mathcal T_\Delta (M_1,M_2,N_1,N_2)$ is defined in Lemma~\ref{BettinChandee}, which we use
  with the following parameters:   
 \begin{itemize}
 \item
 $M_1=M_2=N_1=N_2 =N$, 
 \item  $\alpha_{n_1}$ is the characteristic function of $\mathcal N_1$ and $\beta_{n_2}$ is the characteristic function of $\mathcal N_2$, 
 \item $f (\xi) =g(\xi) = \psi\left( \frac \xi {3N/4}\right)$,  where $\psi$ is a fixed positive function belonging to $\mathcal C^\infty ([2/3, 4/3 ], \mathbb R)$
 such that $\psi (t) =1$ for $5/6 \leq t \leq 7/6$ and such that $0\leq \psi (t) \leq 1$, for all real $t$. 
 \end{itemize}
 
 In particular,  we see from the last condition that for all real $\xi$ and all integer $j\geq 0$
we have the inequality
$$
|f^{(j)} (\xi)|  \leq (4/3)^j N^{-j} \sup_t  |\psi^{(j)} (t)|.
$$  

 With these choices, we have $R=2$ and  $\eta=4/3$.  
  The error term in~\eqref{formulaforTDelta} is $O (N^{\frac{39}{20} +\varepsilon_0})$, where 
 $\varepsilon_0 >0$ is arbitrary. To deal with the main term, we first consider the
 integral. We have

 \begin{lemma}\label{Lebesgue} We adopt the definitions of $f$ and $g$ given above. 
 We also assume that $ 2\leq \Delta \leq N^2/1000$ and that $n_1 \in \mathcal N_1$ and $n_2 \in \mathcal N_2$. Furthermore we suppose that
 \begin{equation}\label{n1/n2}
 9/10\leq n_1/n_2 \leq 10/9.
 \end{equation}
 Then the Lebesgue  measure of the set of real  $x$ such that 
 $$
 f(x+\Delta/n_2) g(x/n_1) =1
 $$ is 
 at least
 $cN^2$, where $c>0$ is independent from $\Delta$, $N$, $n_1$ and $n_2$ as above. 
 \end{lemma}
 \begin{proof} By the definition of the functions $f$ and $g$, we see that the set of $x$ such that $f(x+\Delta/n_2) g(x/n_1)=1$ is included in the set of $x$ satisfying the two conditions
 \begin{equation}\label{740}
 \begin{cases} 
 \frac 58 n_2N -\Delta \leq  x \leq \frac 78 n_2 N -\Delta \\
  \frac 58 n_1N\leq x \leq  \frac 78 n_1 N.
 \end{cases}
 \end{equation}
 By~\eqref{n1/n2}, the second condition of~\eqref{740} is satisfied when one has the inequalities
 $$
\frac 58 \cdot \frac {10}9 \, n_2 N\leq  x \leq  \frac 78 \cdot \frac 9{10}\, n_2 N.
$$
This formula defines  a non-empty segment, which is included in the segment corresponding to the first condition of~\eqref{740}. 
 \end{proof}
 
 Combining Lemma~\ref{Lebesgue} with Lemma~\ref{BettinChandee}, and recalling our above 
observation that  the error term   in~\eqref{formulaforTDelta} is $O (N^{\frac{39}{20} +\varepsilon_0})$,
 we deduce that if $2\leq \Delta \leq N^2/1000$, one has the inequality
 \begin{align*}
 \mathcal T_\Delta(N,N,N,N)\ &\gg N^2 \, \underset {\substack{n_1\in \mathcal N_1,\, n_2 \in \mathcal N_2 \\ 9/10 \leq n_1/n_2 \leq  10/9\\ \gcd(n_1,n_2) \mid \Delta}}
 {\sum \sum} \frac {\gcd(n_1,n_2)}{n_1n_2}  + O\(N^{\frac{39}{20} +\varepsilon_0}\)\\
 & \gg N^2.
 \end{align*}  
 We now return to~\eqref{721} to insert this lower bound. We sum over $\Delta$ with $(\Delta/p) =-1$ to complete the proof by appealing to Lemma~\ref{599}.
 
 \section{Proof of theorem~\ref{thesumU}} 
  \subsection{Preliminaries} 
  Let $\varepsilon >0$ be fixed  and $p ^{1/2} >N\geq p^{1/8 +\varepsilon}$.  We introduce the parameters $\varepsilon_1$ and $\varepsilon_2 >0$. Their values are  small and are given at the end of the proof, in terms of $\varepsilon$ and the constant $\delta$, (depending on $\varepsilon$) which appears in Lemma~\ref{Burgessclassical}.

 Clearly, without loss of generality, we can assume that $\vert  \alpha_a \vert \leq 1$ and $  \vert \beta_b\vert  \leq 1$ for admissible values of $a$ and $b$.   
  
   \subsection{Splitting the summations} The proof of this 
 Theorem~\ref{thesumU} uses Lemma~\ref{BettinChandee} in a crucial way. In order to apply it,  our first
 task is to decompose $U_\chi (\boldsymbol \alpha, \boldsymbol \beta, N ) $ into subsums where the orders  of magnitude of the variables $a$, $b$, $c$ and $d$ are controlled.    Let 
 $$
Y:= N^{\varepsilon_1}  \mand  \Xi := 1 + N^{-\varepsilon_1}=1+Y^{-1}. 
 $$
 We introduce the following finite set of real numbers:
 $$\mathfrak N:=  \left\{ N/2^\nu:~  N/2^\nu \ge 1, \ \nu = 0,1, \ldots \right\},
 $$  
 and  the counting function 
   $$
  \mathcal T_\Delta^* (N) :=  \underset{ \substack{ 1\leq a \leq N,\, 1\leq b \leq N,\, 1\leq c \leq N,\, 1\leq d \leq N \\ ad-bc=\Delta}}{\sum\  \sum \ \sum \ \sum} \alpha_a \beta_b.
  $$
 We decompose $ U_\chi  (\boldsymbol \alpha, \boldsymbol \beta, N )$ as
 \begin{equation} 
 \label{785}\
  U_\chi  (\boldsymbol \alpha, \boldsymbol \beta, N ) = U_\chi^+  (\boldsymbol \alpha, \boldsymbol \beta, N ) + U_\chi^-  (\boldsymbol \alpha, \boldsymbol \beta, N )
   \end{equation} 
   where   
 \begin{align*} 
& U_\chi^+  (\boldsymbol \alpha, \boldsymbol \beta, N ) 
   := \sum_{0< \Delta < N^2 } \chi (\Delta) \mathcal T_\Delta^* (N), \\
& U_\chi^-  (\boldsymbol \alpha, \boldsymbol \beta, N ) := \sum_{ -N^2<  \Delta <0} \chi (\Delta) \mathcal T_\Delta^* (N).
\end{align*} 
 The proof of Theorem~\ref{thesumU}  reduces to the existence of a positive $\delta_0$ such that
\begin{equation}\label{purpose}
U_\chi^+  (\boldsymbol \alpha, \boldsymbol \beta, N ) = O(N^{4 -\delta_0}),
\end{equation}
since the sudy of the sum $ U_\chi^-  (\boldsymbol \alpha, \boldsymbol \beta, N ) $ is similar.

Let us impose the following restriction 
    \begin{equation}\label{firstrestriction}
  \varepsilon_1 \geq \delta_0.
  \end{equation}

We decompose the  characteristic functions of the intervals of variations of $a$ and $b$ as a  sum of characteristic functions $\mathbf 1_{K}$ of intervals $[K/2, K]$ (see Section~\ref{sec:conv}), where $K$ belongs to $\mathfrak N$.   
For the  variables $c$ and $d$, we use the functions $b_{\ell, \Xi}$ introduced in Lemma~\ref{smoothpartition} with  $1\leq \ell \leq L$
where $L$ is the unique  integer satisfying 
\begin{equation}\label{XiL}
\Xi^{L+1} \leq N < \Xi^{L+2}.
\end{equation}
For a positive $\Delta$,  we have the equality
 \begin{equation}
\begin{split} \label{--}
  \mathcal T_\Delta^* (N) =\sum_{A\in \mathfrak  N}  \sum_{B\in \mathfrak N} & 
  \sum_{1\leq \ell, \newell \leq L}   \underset{\substack{1\leq a \leq N,\, 1\leq b \leq N,\, 1\leq c \leq N,\, 1\leq d \leq N \\ ad-bc=\Delta}}{\sum\  \ \sum \ \  \sum \  \ \sum} \\  & \qquad \quad \mathbf 1_{ A} (a) \alpha_a\cdot  \mathbf 1_B (b)\beta_b\cdot  b_{\ell, \Xi}(  c)  \cdot b_{\newell , \Xi} (d) + E(\Delta). 
\end{split}
\end{equation}
The complementary term  $E(\Delta)$ corresponds to the contribution of quadruples $(a,b,c,d)$  with $ad-bc =\Delta$, with $1\leq a, \, b \leq N$  and at least one of the variables $c$ or $d$
is in the interval $[\Xi^L, N]$.  We consider $E (\Delta)$ as an error term since, when   summing over $\Delta$ and appealing to the definition of $L$ (see~\eqref{XiL}) we obtain the bound 
\begin{equation}\label{E(Delta)onaverage}
\sum_{1\leq \Delta < N^2} \vert \chi (\Delta)  E (\Delta)\vert  \leq N^3\( \Xi^{L+2} -\Xi^L\) \ll N^4Y^{-1}. 
\end{equation}   
This error term is acceptable by the restriction~\eqref{firstrestriction}.

We write~\eqref{--}  
\begin{equation}\label{809} 
 \mathcal T_\Delta^* (N) =\sum_{A\in \mathfrak  N} \sum_{B\in \mathfrak N} \sum_{1\leq \ell \leq L} \sum_{1\leq \newell  \leq L}
\mathcal T_\Delta^* (  \Xi^\ell, \Xi^{\newell },A,B) + E (\Delta),
\end{equation}
with obvious notations. 
\subsection{Restricting the sommation over   $\Xi^\ell$, $\Xi^{\newell }$, $A$, $B$} Decompose
the sum in~\eqref{809} into
$$
 \mathcal T_\Delta^* (N) :=  \mathcal T_\Delta^* (N, \leq N^4Y^{-1} ) + \mathcal T_\Delta^* (N, > N^4Y^{-1})+E (\Delta)
 $$
 where
 \begin{itemize} 
 \item $ \mathcal T_\Delta^* (N, \leq N^4Y^{-1} ) $ is the quadruple sum (over   $\Xi^\ell$,  $\Xi^{\newell }$, $A$ and $B$)  restricted by the extra condition $ \Xi^{\ell} \Xi^{\newell } AB\leq N^4/Y$ and 
  \item $ \mathcal T_\Delta^* (N, >N^4Y^{-1} ) $ corresponds to the restriction  $ \Xi^{\ell} \Xi^{\newell }AB >N^4Y^{-1}$.
  \end{itemize}
Changing the order of summation, similarly to~\eqref{E(Delta)onaverage}, we  deduce the  obvious inequality
  \begin{equation}\label{822}
  \sum_{1\leq \Delta<N^2 } \chi (\Delta)  \mathcal T_\Delta^* (N, \leq N^4Y^{-1} )  = O( N^4Y^{-1}).
  \end{equation} 
  This error term is acceptable under the restriction~\eqref{firstrestriction}.
  
  We now concentrate on the sum $ \mathcal T_\Delta^* (N, > N^4Y^{-1} )$.  Actually, the inequalities $  \Xi^\ell,\,  \Xi^{\newell }, \, A, \, B \leq N$ combined with $ \Xi^\ell\, \Xi^{\newell } AB>N^4 Y^{-1}$  imply that  $ \Xi^\ell$,  $\Xi^{\newell }$, $A$ and $B$ are of comparable sizes, which means 
\begin{equation}\label{alllarge}
  NY^{-1} \leq   \Xi^\ell, \,  \Xi^{\newell }, \, A, \, B\leq N.
  \end{equation}
  Furthermore the number $Q$ of such quadruples $(  \Xi^\ell, \Xi^{\newell }, A,B)$ is bounded by
\begin{equation}\label{numberofdissections}
 Q  \ll (\log N)^2 \left( \frac{\log N}{\log (1+Y^{-1})}\right)^2 \ll Y^2 (\log N)^4 .
  \end{equation}
  
 
  \subsection{Bounding the error term} In order to prove~\eqref{purpose}, we are now led to  proving the inequality
 \begin{equation}
\begin{split} \label{835}
 \sum_{1\leq \Delta <N^2 } \chi (\Delta)  \underset{\substack{ 1\leq a \leq N,\, 1\leq b \leq N,\, 1\leq c \leq N,\, 1\leq d \leq N \\ ad-bc=\Delta}}{\sum\   \sum \   \sum \   \sum} \mathbf 1_{ A} & (a) \alpha_a\cdot  \mathbf 1_B (b)  \beta_b 
   \cdot  b_{\ell, \Xi}(  c) \cdot  b_{\newell , \Xi} (d ) 
 \ll N^{4-\delta_0 -3\varepsilon_1}, 
\end{split}
\end{equation}  
  for all $(  \Xi^\ell, \Xi^{\newell }, A, B) $ satisfying~\eqref{alllarge}.
  This is a consequence of ~\eqref{785}, \eqref{E(Delta)onaverage}, \eqref{809}, \eqref{822} 
  and~\eqref{numberofdissections}.   
  
    In~\eqref{835}, the quadruple sum is exactly  $\mathcal T_\Delta (\Xi^\ell, \Xi^{\newell }, A, B)$ 
     which is treated in Lemma~\ref{BettinChandee} and which we apply 
    with the following values
    $$
\begin{cases}
M_1:= \Xi^{\ell +1},\\
M_2: = \Xi^{\newell  +1},\\
N_1:= A,\\
N_2 := B, \\
f(\xi) :=b_{\ell, \Xi} (\xi), \\
g(\xi):= b_{\newell , \Xi}(\xi).
\end{cases}
$$
The corresponding value of $R$, defined in~\eqref{defR},  satisfies the inequality
  $$
  R \leq 2Y^2, 
  $$
  as a consequence of~\eqref{alllarge}. The corresponding value of $\eta$ is 
  $$
  \eta  = \frac{\Xi}{\Xi-1}\leq 2Y.$$
  
  We now fix one of the quadruples $(  \Xi^\ell, \Xi^{\newell }, A,B) $ satisfying~\eqref{alllarge}
  for which we need to estimate the sum $   \mathcal T_\Delta (\Xi^\ell, \Xi^{\newell }, A, B)$.
  
  Next, we appeal to the formula~\eqref{formulaforTDelta} that we write under the form
  \begin{equation}\label{MT+ET}
   \mathcal T_\Delta (\Xi^\ell, \Xi^{\newell }, A, B):= {\rm MT} (\Delta) + {\rm Err} (\Delta).
\end{equation}
  With the above values of the parameters, we see that the contribution of the error term to the left hand side of~\eqref{835} is
 \begin{align*} 
 \sum_{1\leq \Delta <N^2} \chi( \Delta)  {\rm Err} (\Delta)
& \ll  \sum_{1\leq \Delta<  N^2 } (Y^3)^{3/2} A^{1/2} B^{1/2} (AB)^{7/20} (A+B)^{1/4 +\varepsilon_1} (\Xi^{\ell} \Xi^{\newell })^{\varepsilon_1}\\
&\ll  N^{79/20 +9 \varepsilon_1/2 + \varepsilon_1 + 2\varepsilon_1} \ll  N^{79/20 +8 \varepsilon_1}.
\end{align*}  
  This bound fulfils the bound~\eqref{835}, provided that we have 
  \begin{equation}\label{secondrestriction}
  \delta_0 +11 \varepsilon_1  \leq \frac 1{20}.
  \end{equation}
  \subsection{Application of the Burgess bound}
Consider now the contribution of the main term  ${\rm MT}(\Delta)$ (see~\eqref{MT+ET})   to the left hand side of~\eqref{835}. After inverting summations and integration, we see that this contribution is    
 \begin{equation}
\begin{split}
\label{crucial}
\sum_{1\leq \Delta < N^2} \chi (\Delta)  {\rm MT}(\Delta) 
& =  \sum_{ 1\leq a \leq N} \sum_{ 1\leq b \leq N} \mathbf 1_A (a) \alpha_a \cdot \mathbf 1_B (b) \beta_b \cdot \frac {\gcd(a,b)}{ab} \\ 
&\qquad \qquad  \times 
 \int_{\mathbb R} b_{\newell , \Xi} \left(\frac  x a \right) \left( \sum_{\substack{1\leq \Delta < N^2  \\
  \gcd(a,b) \mid \Delta}} \chi (\Delta) b_{\ell, \Xi}\left( \frac{x+\Delta} {b}\right)\right) dx. 
\end{split}
\end{equation}    
  We want to exploit the oscillations of the character $\chi (\Delta)$ by using Lemma~\ref{Burgessclassical}. However, this requires several technical preparations. First of all, in the integral in~\eqref{crucial}, we can reduce  the variable $x$ to the interval
  $$
  \Xi^{\newell -1} \leq x/a \leq  \Xi^{\newell +1} ,
   $$
  which implies that 
  \begin{equation}\label{sizeofx}
   x \leq N^2.
   \end{equation}
   We  now establish the following estimate. 
   
   \begin{lemma} \label{byparts} Let $\chi$ be a non-principal character modulo $p$. Let $\varepsilon >0$ and let $\delta$ be defined by Lemma~\ref{Burgessclassical}.
  We have the bound
   \begin{equation*}
   \sum_{\substack{1\leq \Delta < N^2  \\ t \mid \Delta}}  \chi (\Delta)  b_{\ell, \Xi}\left( \frac{x+\Delta} {b}\right)
   \ll \min  
   \left( N^2t^{-1}, tY^3N^{-2} p^{1/2 +2\varepsilon} + t^{-1}Y^3N^{2} p^{-\delta}\right)
   \end{equation*}
  under the conditions~\eqref{alllarge},  uniformly for $t\geq 1$, for $x$ satisfying~\eqref{sizeofx} and for $N \geq p^{1/8 + \varepsilon}$.
   \end{lemma}
   \begin{proof}
   The first part of the above bound is trivial. For the second one, we put $\Delta = t \newDelta  $. By multiplicativity  we are led to bound the sum 
   $$
  S:= \sum_{1\leq \newDelta < N^2 t^{-1}}  \chi (\newDelta  )\,  b_{\ell, \Xi}\left( \frac{x+t\newDelta  } {b}\right).
   $$
 We plan to use Abel summation.   
 First we observe that by Lemma~\ref{smoothpartition} and since the quadruple $(\Xi^\ell, \Xi^{\newell}, A,B) $ satisfies~\eqref{alllarge}, 
 the derivative of the function $$\newDelta   \mapsto b_{\ell, \Xi}\left( \frac{x+t\newDelta  } {b}\right)$$  is bounded by
   $$
   O\(\frac tb \cdot  Y \Xi^{-\ell}\) = O\( t Y^3 N^{-2}\). 
   $$ 
   So we have the inequality
   \begin{align*}
   S & \ll tY^3N^{-2}\left( \int_1 ^{p^{1/4+\varepsilon}}  + \int_{p^{1/4 +\varepsilon}}^{N^2t^{-1} }\right) \left(\sum_{1\leq \newDelta   \leq u} \chi (\newDelta) \right)du\\
   & \ll tY^3N^{-2} \left\{ p^{1/2+ 2\varepsilon}     +N^{4}t^{-2} p^{-\delta}\right\},
   \end{align*}
   where $\delta>0$ is  defined in  Lemma~\ref{Burgessclassical}  and  depends on $\varepsilon$. 
   \end{proof}
   We return to~\eqref{crucial} and  insert the bound given by Lemma~\ref{byparts}. Let 
   $$
   T= N^{\varepsilon_2}.$$
   Writing $t = \gcd(a,b)$, $a= \newa t  $ and $b =\newb t$ to  remark that the contribution to $\sum_{1\leq \Delta < N^2} {\rm MT}(\Delta)$ of the $(t,  \newa, \newb)$  with $t\geq T $ is negligible, by using the bound $N^2t^{-1}$ for $S$ and~\eqref{sizeofx} for the length of the interval of integration.
   Indeed this contribution  $\fA_1$ can be bounded as follows
\begin{equation}
\label{eq:A1}
 \fA_1  \ll
   \sum_{  t \geq T}  \sum_{A/(2t)\le  \newa \le A/t} \sum_{B/(2t) \le  \newb\le B/t} \frac{1}{t \newa \newb}  N^2 (N^2t^{-1}) \ll N^4 T^{-1}. 
\end{equation}      

   The bound~\eqref{eq:A1} on  $\fA_1$ is compatible with the desired bound~\eqref{835} as soon as we have the inequality
   \begin{equation}\label{N4T-1}
   \delta_0 +3\varepsilon_1 \leq \varepsilon_2.
   \end{equation}
   
To deal with the contribution of the $(t, \newa,\newb)$ with $1\leq t \leq T$, we use the second part of the upperbound given by  Lemma~\ref{byparts}. More precisely this    contribution $\fA_2$  can be bounded as follows
\begin{equation*}
\fA_2  \ll   \sum_{t\leq T} \sum_{A/(2t)\le  \newa \le A/t} 
 \quad \sum_{B/(2t) \le  \newb\le B/t} \frac{1}{t \newa \newb} \, N^2  \left(  t Y^3 N^{-2} p^{1/2 +2\varepsilon} +   t^{-1} Y^3 N^2 p^{-\delta} \right).
\end{equation*}
Therefore
\begin{equation}
\label{end}
\fA_2   \ll T Y^3 p^{1/2 +2 \varepsilon} + N^4 Y^3 p^{-\delta}
  \ll TY^3 N^4 p^{-2\varepsilon} + N^4 Y^3 p^{-\delta}\ll  TY^3 N^{4- 4\varepsilon} +  Y^3 N^{4-2\delta},  
\end{equation}    
 by the assumption $p^{1/2} \geq  N\geq p^{1/8+\varepsilon}$. The bound~\eqref{end} on  $\fA_2$ is compatible with the desired bound~\eqref{835} as soon as one has the two inequalities
 \begin{equation}\label{endofend}
 \delta_0+6\varepsilon_1+\varepsilon_2 \leq 4 \varepsilon \text{ and } \delta_0 + 6 \varepsilon_1 \leq 2 \delta.
 \end{equation}
 \subsection{Choice of the parameters} Recall that when $\varepsilon >0$ is given, the  parameter $\delta >0$    has a fixed value. To complete the proof of Theorem~\ref{thesumU}, it remains to  find positive values for $\varepsilon_1$ and $\varepsilon_2$
and $\delta_0$   such that  the inequalities~\eqref{firstrestriction},   \eqref{secondrestriction}, \eqref{N4T-1} and~\eqref{endofend} are satisfied. By choosing \ 
$$
 \begin{cases}
 \varepsilon_1 = \min \left( 1/240, 4\varepsilon/11, 2\delta/7 \right),\\
 \varepsilon_2 = 4 \varepsilon_1,\\
 \delta_0=\varepsilon_1,
 \end{cases}
 $$
 we complete the proof of Theorem~\ref{thesumU}.

\section{Proof of Theorem~\ref{thm:SumTD1} }\label{section3}
Factoring out $\chi(c)$, we write the equality
\begin{equation}\label{T=Tn0+}
T_\chi( A,B,C, \cD;\balpha)  =  \sum_{1\leq a \leq A} \, \sum_{1\leq b \leq B} \, \sum_{1\leq c \leq C} \,   \left| \sum_{d \in \cD} \alpha_d \chi(ab/c-d)\right| .
\end{equation}
Let $I$ be the counting function defined, for any integer  $1\leq \lambda\leq p-1$,  by the formula 
$$
I(\lambda) = \sharp\{(a,b,c) \in [1,A]\times [1, B]\times [1,C] :~ab/c \equiv \lambda \bmod p\}.
$$
Therefore, we can write~\eqref{T=Tn0+} as
\begin{align*}
T_\chi ( A,B,C, \cD;\balpha)   & =   \sum_{\lambda=1}^{p-1} I(\lambda) 
 \left| \sum_{d \in \cD} \alpha_d \chi(\lambda-d)\right|\\
 & =  \sum_{\lambda=1}^{p-1} \,   \left| \sum_{d \in \cD} \alpha_d \chi(\lambda-d)\right|\cdot  I(\lambda)^{1-1/\nu}  \cdot I (\lambda)^{1/\nu},
\end{align*}
where $\nu \geq 1$ is an integer.   

We apply H\" older's inequality to this trilinear form with the choice of exponents  
$$
\(\frac 1{2\nu}, 1 - \frac 1{\nu}, \frac 1{2\nu}\),
$$ 
leading to the inequality
\begin{equation}
\label{538}
\begin{split} 
T_\chi( A,B,C, \cD;\balpha)^{2\nu}  
&\leq  \sum_{\lambda=1}^{p-1} \,   \left| \sum_{d \in \cD} \alpha_d \chi(\lambda-d)\right|^{2\nu} \cdot 
\left( \sum_{\lambda =1}^{p-1} I(\lambda) \right)^{2\nu -2}  \cdot \left( \sum_{\lambda =1}^{p-1} I(\lambda)^2 \right) \\
&\leq \Sigma_1 \cdot \Sigma_2 \cdot \Sigma_3, 
\end{split}
\end{equation}
with an obvious definition of $ \Sigma_1$,  $\Sigma_2$ and  $\Sigma_3$.

Lemma~\ref{lem: D-E Lem*1} gives the following  bound for $\Sigma_1$
\begin{equation}\label{Sigma1<}
\Sigma_1 \ll (2\nu D)^\nu p +2\nu D^{2\nu} p^{1/2}.
\end{equation}
 For $\Sigma_2$ we use the trivial remark that
\begin{equation}\label{Sigma2<}
\Sigma_2 = (ABC)^{2\nu-2}.
\end{equation}
For $\Sigma_3$, we  benefit from the inequality  $ABC <p$, and  write 
\begin{align*}
\Sigma_3&= \sharp \Bigl\{  (a_1,a_2,b_1,b_2,c_1,c_2)\in [1,A]^2 \times [1,B]^2 \times [1, C]^2:
  a_1b_1c_2 = a_2b_2c_1\Bigr\}
 \leq \sum_{1\leq m \leq ABC} \tau_3^2 (m),
\end{align*}
which finally gives
\begin{equation}\label{Sigma3<}
\Sigma_3 \ll ABC (\log p)^8
\end{equation}
by Lemma~\ref{lem:tau31}. 
Combining~\eqref{538},  \eqref{Sigma1<}, \eqref{Sigma2<} and~\eqref{Sigma3<}, we complete the proof of Theorem~\ref{thm:SumTD1}.
\section{Proof of Theorem~\ref{thm:SumTN1}}
\subsection{Preliminary transformations} 
\label{sec:prelim}
 As an  immediate consequence of Corollary~\ref{cor:SumTD1}, we know that the statement of Theorem~\ref{thm:SumTN1} is correct  if one
imposes the extra condition $p >N > p^{1/6 +\varepsilon}$. So, 
  without loss of generality we can assume that 
\begin{equation}
\label{eq: small N1}
N \le 0.5\,  p^{1/4}. 
\end{equation}

We fix some $\varepsilon > 0$ and observe that by Corollary~\ref{cor:SumTD1} we can 
assume that $\varepsilon < 1/20$.  We also set 
\begin{equation}
\label{eq:choice of eta1}
\eta =    \frac{4 \varepsilon^{-1}  \log \log p}{\log p}.
 \end{equation}  
Then we define 
\begin{equation}
\label{eq:choice of xy1}
x = p^{\eta} = \left(\log p\right)^{ 4/ \varepsilon  } \mand y = p^{\varepsilon} .
 \end{equation}   
 
Next, for an integer $r \ge 0$ we
denote by $\cA_r(N;x,y)$ the set of positive integers $d \le N$
which have exactly $r$ prime divisors (counted with multiplicities) in the
 half-open interval $\cI = (x,y]$.
In particular,  the cardinality of $\cA_0(N;x,y)$ has been  estimated in  Lemma~\ref{lem:ANxy1}.
Let $R$ be the largest value of $r$ for which $\cA_r(N;x,y) \ne \emptyset$.  In particular, we have the
trivial bound  
\begin{equation}
\label{eq: R triv1}
R \ll \frac{\log N}{\log x}.
\end{equation}  

We now write the formula of decomposition 
$$
T_\chi( N) = \sum_{1 \le a,b,c\le N}  \left| \sum_{1\le d\le N}  \chi(ab/c-d)\right|
 \leq \sum_{r =0}^R  U_r,
$$
where
$$
U_r =  \sum_{1 \le a,b,c\le N}  \left|
 \sum_{d \in \cA_r(N;x,y)}  \chi(ab/c-d)\right|. 
$$
Estimating $U_0$ trivially as
$$
U_0 \leq N^3 \cdot ( \sharp\cA_0(N;x,y) ), 
$$
and using Lemma~\ref{lem:ANxy1}, we obtain the inequality 
\begin{equation}
\label{eq: T Ur U01}
T_\chi( N)  \ll  \sum_{r =1}^R  U_r +  N^4 \frac{\log x}{\log y}
\ll   \sum_{r =1}^R  U_r +  \eta N^4.
\end{equation}
Clearly there are at most $N/(x\log 2x)$  integers 
$$
d \in  \bigcup_{r=1}^R \cA_r(N;x,y) ,
$$
which are divisible by a square of a prime $p \in \cI$. This is a consequence of 
Lemma~\ref{Cheby}. Hence we can rewrite~\eqref{eq: T Ur U01} as 
\begin{equation}
\label{eq: T Ur* U0 N4/x1}
T_\chi( N)  \ll  \sum_{r =1}^R  U_r^* +  \eta N^4  + N^4/(x\log 2x), 
\end{equation}
with 
$$
U_r^* =  \sum_{1 \le a,b,c\le N}  \left|
 \sum_{d \in \cA_r^*(N;x,y)}  \chi(ab/c-d)\right|, 
$$
where the set $\cA_r^*(N;x,y) \subseteq \cA_r(N;x,y)$ 
is defined as  the set of positive integers $d \le N$
which have exactly $r$  distinct prime divisors $p\in \cI$.
 
Clearly, every integer $d \in \cA_r^*(N;x,y)$,
has exactly $r$ representations as $d = \ell m$
with a prime $\ell \in \cI$  and integer $m \in \cA_{r-1}^*(N/\ell;x,y)$.
Hence, for $r =1, \ldots, R$, we have 
\begin{align*}
U_r^* & = \frac{1}{r}  \sum_{1 \le a,b,c\le N}  \left| \,
 \sum_{\ell \in \cI}  \sum_{\substack{m\in \cA_{r-1}^*(N/\ell;x,y)\\ \gcd(\ell,m)=1}} \chi(ab/c-\ell m)\, \right|\\
&   = \frac{1}{r}  \sum_{1 \le a,b,c\le N}  \left| \,
 \sum_{\ell \in \cI} \  \sum_{\substack{m\in \cA_{r-1}^*(N/\ell;x,y)\\ \gcd(\ell,m)=1}} \chi (m) \ 
  \chi\(ab/(c m)-\ell\)\, \right|, 
\end{align*}
where throughout the proof, $\ell$ always denotes a prime number. Changing the order of summation, we now write the inequality
\begin{equation}
\label{eq:UrVr1}
U_r^* \leq  \frac{1}{r} V_r, 
\end{equation}
where,  
$$
V_r : =   \sum_{1 \le a,b,c\le N}  \sum_{m\in \cA_{r-1}(N/x ;x,y)}  
   \left|\sum_{\substack{\ell \in \cI\cap[1, N/m]\\ \gcd(\ell,m)=1}}
 \chi\(ab/(c m)-\ell\)\right| .
$$
Let
$V^\dag_r$ the following modification of
 $V_r$:
 $$
 V^\dag_r  :=\sum_{1 \le a,b,c\le N}  \sum_{m\in \cA_{r-1}(N/x ;x,y)}  
   \left|\sum_{\substack{\ell \in \cI\cap[1, N/m] }}
 \chi\(ab/(c m)-\ell\)\right| 
 $$
 To control the gap between $V_r $ and $V_r^\dag$, we introduce
 $$
E_r :=  \vert V_r -V_r ^\dag \vert \ll N^3 \sum_{x < \ell \leq y } \sum_{\substack{m\in \mathcal A_r (N; x,y) \\ \ell^2 \mid m}} 1. 
 $$
 Summing over all possible $r \leq R$, we have the following upper bound  
 \begin{equation}\label{sumEr<<}
 \sum_{r\leq R}  E_r \ll N^4 /(x \log x),
 \end{equation}
 by Lemma~\ref{Cheby}. 
 
 Furthermore, in the definition of $V_r^\dag$,  we also replace the  condition 
$m\in \cA_{r-1}(N/x ;x,y)$   by $m\leq N/x$. We split this interval into  $m \le N/y$ 
and $N/y  <m \le  N/x$, which of course simply increases the sum. Thus,  from the triangular inequality $V_r \leq V_r^\dag +E_r$, we derive 
\begin{equation}
\label{eq:VrWr121}
V_r \ll  W_{1} +  W_{2}  +  E_r.   
 \end{equation}
where 
\begin{align}
W_{1}  & =   \sum_{1 \le a,b,c\le N}  \sum_{m \le N/y}   
   \left|\sum_{\ell \in \cI} \chi\(ab/(c m)-\ell\)\right|, \label{defW1}\\
W_{2}  & =   \sum_{1 \le a,b,c\le N}  \sum_{N/y < m\leq  N/x}
   \left|\sum_{\ell \in [1, N/m]}
 \chi\(ab/(c m)-\ell\)\right|.\label{defW2}
\end{align}

Note that we dropped the subscript $r$ in the sums $W_1$ and $W_2$ as they
no longer depend on $r$. 
Gathering~\eqref{eq: T Ur* U0 N4/x1}, \eqref{eq:UrVr1}, \eqref{sumEr<<} and~\eqref{eq:VrWr121},  we proved the inequality
\begin{equation}\label{779}
T_\chi (N) \ll  (W_1+W_2) \left( \sum_{r\leq R} \frac 1r \right) +   \eta N^4  + N^4/(x\log 2x). 
\end{equation}
The purpose of the next two sections is to give suitable bounds for $W_1$ and $W_2$. 

\subsection{Study of the sum $W_{1}$} \label{sec:W11}
To estimate $W_{1}$ defined in~\eqref{defW1}, we collect together the ratios $ab/(c m)$ which fall in the same residue class modulo $p$ and define
\begin{equation}
\label{eq:J-lambda1}
J(\lambda) = \sharp\{(a,b,c,m) \in [1,N]^3\times   [1,N/y]:
  ab/(c m) \equiv \lambda \bmod p\}.
 \end{equation}  
Hence 
$$
W_{1}   =   \sum_{\lambda=1}^{p-1} J(\lambda) 
   \left|\sum_{\ell \in \cI} \chi\(\lambda-\ell\)\right|.
$$
We now see that the H{\"o}lder inequality (with the same exponents as in \S\ref{section3})   and Lemma~\ref{lem: D-E Lem*1}
(after discarding the primality condition on $\ell$ to simplify the final expression, 
which is inconsequential for the final result) yields the inequality 
\begin{equation}
\label{eq:Wr1-Holder-11}
W_{1}^{2\nu}  \le   \(\(2\nu y\)^\nu p + 2\nu y^{2\nu} p^{1/2}\)\(  \sum_{\lambda=1}^{p-1} J(\lambda) \)^{2\nu-2} 
 \sum_{\lambda=1}^{p-1} J(\lambda)^2 . 
 \end{equation}
  Recalling the choice of $y$ in~\eqref{eq:choice of xy1}, we see that with
\begin{equation}\label{defnu}
\nu = \rf{\varepsilon^{-1}},
\end{equation}
we have $y^\nu p \ll  y^{2\nu} p^{1/2}$ and~\eqref{eq:Wr1-Holder-11} becomes 
\begin{equation}
\label{eq:Wr1-Holder-21}
W_{1}^{2\nu}  \ll y^{2\nu} p^{1/2}\(  \sum_{\lambda=1}^{p-1} J(\lambda) \)^{2\nu-2} 
 \sum_{\lambda=1}^{p-1} J(\lambda)^2  . 
 \end{equation}
 
We clearly have 
\begin{equation}
\label{eq:J11}
\sum_{\lambda=1}^{p-1} J(\lambda) \ll N^4/y
 \end{equation}
while using condition~\eqref{eq: small N1} we derive that 
\begin{align*}
\sum_{\lambda=1}^{p-1}  J(\lambda)^2= \sharp\{(a_1,a_2,b_1,b_2,c_1,c_2,m_1,m_2)
 \in [1,N]^6\times  [1,N/y]^2: 
 a_1b_1 c_2 m_2  = a_2b_2 c_1 m_1\}.
 \end{align*}
Lemma~\ref{lem:tau31} with $s=4$,  now implies the inequality
\begin{equation}
\label{eq:J21}
 \sum_{\lambda=1}^{p-1} J(\lambda)^2 \le \left(  N^4/y \right) \,\log^{15} N.
 \end{equation}
 Hence, substituting the bounds~\eqref{eq:J11} and~\eqref{eq:J21}
 in~\eqref{eq:Wr1-Holder-21} we obtain 
$$
W_{1}^{2\nu}  \ll N^{8\nu - 4} p^{1/2} y \, (\log p)^{15} \ll   N^{8\nu} p^{- 3\varepsilon}\, (\log p)^{15},
$$
by the hypothesis~\eqref{N>p1/8}. 
 Therefore 
$$
W_{1} \ll    N^{4} p^{- 3\varepsilon /2\nu} \, (\log p)^{15/2\nu}.
$$
  Recall that $\nu$ is defined in~\eqref{defnu}, so this bound has the shape
    \begin{equation}
\label{eq:Wr1-fin1}
W_{1} \ll    N^{4} p^{-c_0 \varepsilon^2},
 \end{equation}
 for some absolute positive $c_0$.
  The bound~\eqref{eq:Wr1-fin1} is sufficient for our purpose,
as it is shown by~\eqref{779} and the value of $R$.

 \subsection{Study of the sum $W_{2}$} 
Clearly the sum $W_{2}$, defined in~\eqref{defW2}, is less than the  sum of $O(\log p)$ sums of the form
$$
W_{2}(z)=     \sum_{1 \le a,b,c\le N}  \sum_{N/z  < m \le 2N/z}
   \left|\sum_{\ell \in [1, N/m]}
 \chi\(ab/(c m)-\ell\)\right|. 
$$
with some integers $z \in [x,y]$.
This implies that there exists  some integer $z \in [x,y]$ for which one has the inequality 
\begin{equation}
\label{eq: WrWz1}
W_{2}\ll W_{2} (z) \log p.
\end{equation}
Using the standard completing technique (see~\cite[Section~12.2]{IwKow}),  we can make the range of summation over $\ell$ independent of $m$, by appealing to the additive character $\e$.  More precisely,  for every $N/z< m \leq 2N/z$,   we have the equalities  
\begin{align*}
\sum_{\ell \in [1, N/m]} 
 \chi\(ab/(c m)-\ell\) 
 & = \sum_{\ell \in [1, z]}
 \chi\(ab/(c m)-\ell\)\left(\,  \frac{1}{z} \sum_{h=0}^{z-1} \sum_{r \in [1, N/m]} \e\( h(\ell-r)/z\)\,\right)\\
  & =  \frac{1}{z}  \sum_{h=0}^{z-1}  \sum_{r \in [1, N/m]} \e\(-h r/z\)   \sum_{\ell \in [1, z]}
 \chi\(ab/(c m)-\ell\) \e \(h \ell/z\). 
 \end{align*}
 
 Using~\cite[Bound~(8.6)]{IwKow}, we now derive 
\begin{equation}
\label{eq: WzWhz1}
W_{2} (z) \ll \sum_{h = -z}^{z} \frac{1}{|h|+1} W_{2} (h; z) , 
\end{equation}
where 
$$
W_{2} (h; z) =     \sum_{1 \le a,b,c\le N}  \sum_{N/z \le m \le 2N/z}
   \left|\, \sum_{\ell \in [1, z]}
 \chi\(ab/(c m)-\ell\)\e \(h \ell/z\)\, \right|. 
$$
We now proceed as in Section~\ref{sec:W11}, however instead of~\eqref{eq:J-lambda1}
we  now define
$$
J(z,\lambda) = \sharp\{(a,b,c,m) \in [1,N]^3\times[1,2N/z]:~ab/(cm) \equiv \lambda \bmod p\},
$$
(note that at this point  we discard the lower bound $m \ge N/z$).
With this notation, we have 
$$
W_{2} (h; z) \le     \sum_{\lambda=1}^p J(z,\lambda) 
   \left|\, \sum_{\ell \in [1, z]}
 \chi\(\lambda -\ell\)\e \(h \ell/z\)\, \right|. 
$$
We have the following full analogues of~\eqref{eq:J11}  and~\eqref{eq:J21}  (with $z$ instead of $y$):
\begin{equation}
\label{eq:Jz11}
\sum_{\lambda=1}^{p-1} J(z,\lambda) \ll N^4/z
 \end{equation}
and
\begin{equation}
\label{eq:Jz21}
 \sum_{\lambda=1}^{p-1} J(\lambda)^2 \ll \left( N^{4}/z \right)\log^{15} N.
 \end{equation}
Using  Lemma~\ref{lem: D-E Lem*1} and the bound~\eqref{eq:Jz11} and~\eqref{eq:Jz21}, 
similarly to~\eqref{eq:Wr1-Holder-11}  (with $z$ instead of $y$) 
we derive 
\begin{equation}\label{903}
W_{2} (h; z) ^{2\nu}  \ll  \(N^4/z\)^{2\nu-1} 
 \(\(2\nu z\)^\nu p + 2\nu z^{2\nu} p^{1/2}\) \log^{15} p.
\end{equation}
However, we emphasize that we now take the parameter $\nu$ to be a function of $z$ 
and thus it is not uniformly bounded anymore. Hence we include 
the dependence on $\nu$ in our estimates.

Using the crude bound   
\begin{equation}
\label{eq:crude1}
\(\(2\nu z\)^\nu p + 2\nu  z^{2\nu} p^{1/2}\)
\le  \(2\nu \)^\nu  \( z^\nu p +  z^{2\nu} p^{1/2}\),
 \end{equation}
 and choosing 
$$
\nu = \fl{\frac{\log p}{\log z}}
$$
so the last term in~\eqref{eq:crude1} dominates. Returning to~\eqref{903}, 
we obtain the inequality
\begin{align*}
W_{2} (h; z) ^{2\nu}  &  \ll    \(2\nu \)^\nu\(N^4/z\)^{2\nu-1} z^{2\nu} p^{1/2} \log^{15} p \\
& =
  \(2\nu \)^\nu N^{8\nu - 4} p^{1/2} z \log^{15} p.  
 \end{align*}
 Using that $z \le y$ and recalling~\eqref{N>p1/8} and ~\eqref{eq:choice of xy1}
 (and then  replacing   $\log^{15} p $ with  $p^{\varepsilon}$) we obtain 
$$
 W_{2} (h; z) ^{2\nu}  \ll  \(2\nu \)^\nu N^{8\nu} y p^{- 4\varepsilon } \log^{15} p
 \ll   \(2\nu \)^\nu N^{8\nu}  p^{- 2\varepsilon}.
$$
Therefore we have  the inequality  
\begin{equation}
\label{eq: Whz-11}
W_{2} (h; z)   \ll  \nu N^{4} p^{- \varepsilon/\nu}.
\end{equation}
Since $z \ge x$ we see that 
$$
\nu \le  \frac{\log p}{\log x} = 1/\eta
$$ 
and recalling the choice of $\eta$
in~\eqref{eq:choice of eta1} we obtain the inequalities 
$$
 \nu \ll \log p \mand 
 p^{\varepsilon/\nu} \ge   p^{\varepsilon\eta}  =  (\log p)^{4}.
$$
Thus the bound~\eqref{eq: Whz-11} implies that,  for any admissible $h$ and $z$,  we have
\begin{equation}
\label{eq: Whz-21}
W_{2} (h; z)   \ll  N^{4}  (\log p)^{-3}.
\end{equation}
Substituting~\eqref{eq: Whz-21} in~\eqref{eq: WzWhz1} we obtain 
$$
W_{2} (z)   \ll  N^{4}  (\log p)^{-2}, 
$$
which together with~\eqref{eq: WrWz1}  gives 
  \begin{equation}
\label{eq:Wr2-fin1}
W_{2}  \ll  N^{4}  (\log p)^{-1}.
 \end{equation}
 
\subsection{Concluding the proof}
Substituting the bounds on $W_1$ and $W_2$ given by~\eqref{eq:Wr1-fin1} and~\eqref{eq:Wr2-fin1}, respectively, into~\eqref{779} we obtain 
  the bound
 \begin{equation}
\label{eq:GrandFinale}
T_\chi (N)  \ll
 \log R \left( N^4 p^{-c_0 \varepsilon^2} + N^4 (\log p)^{-1} \right) 
    + \eta N^4 + N^4/(x\log 2x). 
 \end{equation}
It suffices to recall the value of $R$ (see~\eqref{eq: R triv1}) and $\eta$ and $x$ (see~\eqref{eq:choice of eta1} and~\eqref{eq:choice of xy1}) to complete the proof of  
Theorem~\ref{thm:SumTN1}.

\section{Comments} 

It is easy to know that in our decomposition in \S\ref{sec:prelim}
we can discard integers $d \in \cA_r(N;x,y)$ for an abnormally large $r$ 
and at the cost of a small error term,  truncate the summation in~\eqref{eq: T Ur* U0 N4/x1}
for $R$ replaced say with  a smaller value:
$$
R_0:=100 \log \log N,
$$
Indeed, the contribution of the $n$ which have at least $R_0$ prime factors is negligible since if $\omega (n) \geq R_0$, then
$\tau (n) \geq 2^{R_0}$ and from $\sum_{n\leq N}\tau (n) \sim N\log N$, we deduce that the number of $ \leq N$ with $\omega (n) \geq R_0$
is bounded by the order of magnitude by $N 2^{-R_0}  ( \log N)\, \ll N (\log N)^{-2} $.
Unfortunately this does not improve our final result since the bottleneck comes 
from the choice of $\eta$ and $x$ in our optimisation of~\eqref{eq:GrandFinale}. 

We also note that the bounds of Theorem~\ref{thm:SumTN1} also applies to more general sums 
with $a^{\pm 1} b^{\pm 1} -  u c^{\pm 1}d^{\pm 1}$ for a fixed integer $u \not \equiv 0 \bmod p$ 
and any fixed choice of signs in the exponents.

There are two natural generalisations of our question concerning matrices without square root.

One may  think about the non-solvability of the equation  $A=X^d$ (instead of $A=X^2$), with $X\in \mathcal M_2^* (p)$,
where $d \ge 2$ is a divisor of $p-1$.   Although our present method applies we remark
that for large $d$ several other effects take place in the distribution of $d$-th power non-residues 
as it has been known  since the work Vinogradov~\cite{Vi}, see also~\cite{DaEr,Poll}.

Another interesting direction is towards multidimensional analogues of 
our work, that is, for $n\times n$ invertible matrices modulo $p$ with small entries, 
which are not squares in $\mathcal M_n^*( p)$ (defined fully analogously to 
$\mathcal M_2^*( p)$). 

Finally, we noticed that the ``decomposition'' idea, used in the proof of 
Theorem~\ref{thm:SumTN1} takes its roots in the work of Korolev~\cite{Kor}. Of course, the literature
contains other examples of decomposition techniques to create bilinear forms by force 
(for instance~\cite[p.560]{Fo-Ra}),
but we preferred to follow Korolev's idea~\cite{Kor} because of its elegance.



\section*{Acknowledgements}The authors are very grateful to Satadal Ganguly for posing the question,
 which has been the starting point of this work.  We also thank him for his remarks concerning a previous version of this paper.
 
 This work started during a very enjoyable visit by I.S.  to  
 Institut de Math{\'e}matiques de Jussieu whose hospitality 
 and support are very much appreciated. 
During the preparation of this work I.S. was also supported in part by the  
Australian Research Council Grant  DP200100355.




\appendix
\section{Matrices with square roots} 
\label{app:A2}
 We now present some simple arguments which allow to 
 evaluate $\sharp \mathcal M_2^{\ne  \square}(p)$ explicitly. In fact, just to show the ideas we compress the details related to some exceptional case into the error term instead of calculating their precise contribution (which can be easily accomplished). 
It is easier to work with the complementary  set 
$$
\mathcal M_2^{\ddag} (p):= \mathcal M_2 (p) \setminus \mathcal M_2^{\ne  \square}(p).
$$
We have the decomposition of $\mathcal M_2^{\ddag} (p)$  into two disjoint subsets 
\begin{equation}\label{1780}
\mathcal M_2^{\ddag   }(p)= \{ A \in \mathcal M_2 (p) : A =B^2 \text{ for some } B \in \mathcal M_2 (p)\} \cup \mathcal E (p),
\end{equation}
where $\mathcal E (p) \subseteq \mathcal M_2 (p)$ is such that $\sharp \mathcal E (p) = O (p^3)$. This bound directly follows from the classical one: $$\sharp \mathcal M_2 (p) -\sharp \mathcal M_2^* (p) = O (p^3).$$  Now write~\eqref{1780} as
\begin{equation}\label{1784} 
\mathcal M_2 ^ \ddag (p) :=  \mathcal M_2^{\square}(p)\cup \mathcal E (p).
\end{equation}

Consider the equation 
 \begin{equation}
\label{eq:AB2}
A= B^2
 \end{equation} 
where $A \in \mathcal M_2 (p)$ is given and where the unknown is $B\in  \mathcal M_2 (p)$.
Clearly 
$$
(\det B)^2 = \det A
$$
hence for some $s\in \F_p$ we must have $\det A = s^2$ and thus $\det B = \pm s$.
From  the characteristic equation of $B$
$$
B^2 -  t  B + \det B I_2 = 0
$$
where $t = \Tr B$ and $I_2 \in  \mathcal M_2 (p)$ is the identity matrix,  using~\eqref{eq:AB2},  
we derive $A - t B   \pm s  I_2= 0$ 
and hence 
 \begin{equation}
\label{eq:tBAI}
t B = A \pm s  I_2 .
 \end{equation}
Taking the trace again we arrive to 
$$
t^2 = \Tr A \pm  2s.
$$
We now discard $O(p)$ diagonal matrices $A$ of the form $A = uI_2$, $u \in \F_p$. 
For the remaining matrices we see from~\eqref{eq:tBAI} that $t\ne 0$. 
On the other hand when $s$ and $t$ are fixed then the equation~\eqref{eq:tBAI}
defines  $B$ uniquely and $B^2=A$.
Hence we now conclude that for all but $O(p)$ quadruples $(a,b,c,d) \in \F_p^4$,  
we have
$$
\begin{pmatrix} a & b \\ c  & d \end{pmatrix} \in \mathcal M_2^{\square} (p)
$$
if and only if for some for some $(s,t) \in \F_p \times \F_p^*$   
$$ad- bc = s^2 
\mand a+d + 2s = t^2
$$
We obviously have $p^2/2 +O(p)$ choices for the pair $(s,t^2)$. Denoting $u:=s^2$ and $v: = t^2-2s$
we obtain 
 \begin{equation}
\label{eq:Systeq}
ad-bc =u 
\mand a+d  = v
 \end{equation}

 We now count the number of pairs $(u,v)$ obtained this way, that is the image size of the 
 map $(s,t^2) \mapsto (s^2, t^2-2s)$.

%

We now consider the following two types of pairs  $(s,t^2)$:

\begin{description}
\item[A]  Pairs  $(s,t^2)$ such that $4s=t^2-w^2$ for some  $w\in \F_p^*$. Note that if $(s, t^2)$ is of this type then 
so is $(-s, w^2)$ since $-4s=w^2-t^2$. Furthermore, $(s,t^2)$  and $(-s, w^2)$ are mapped to the same pair 
$(u,v) =  (s^2, t^2-2s) =  ((-s)^2, w^2-2(-s))$ and in fact  they are the only pre-images of this pair $(u,v)$.

\item[B] Pairs  $(s,t^2)$ such that $4s\ne t^2-w^2$ for any  $w\in \F_p^*$. Then one easily checks that in this case
 the pair $(s, t^2)$ is the only one which is 
mapped to $(u,v) = (s^2, t^2-2s)$. 
\end{description}

Clearly the above cases~{\bf{A}} and~{\bf{B}} contains $p^2/4+O(p)$ pairs each, thus they 
 contribute $p^2/8+O(p)$  and $p^2/4+O(p)$  pairwise distinct 
pairs $(u,v)$, respectively. Hence,  in total we obtain $3p^2/8+O(p)$  pairwise distinct pairs $(u,v)$. 
Thereby, the elements of $\mathcal M_2^{\square}(p)$ correspond to 
$3p^2/8+ O(p)$ distinct systems of equations~\eqref{eq:Systeq}. Forming a quadratic equation
$$
x^2  -v x + u+bc = 0
$$
for $(a,d)$ we see that it has a solution if its discriminant $\Delta =v^2 -4(u+bc)$ is a square in $\F_p$.
This gives $(p+1)/2$ possibilities for the product $bc$ and hence $p^2/2 + O(p)$ values for the pair
$(b,c) \in \F_p^2$. There are $O(p)$ pairs corresponding   $\Delta = 0$ 
for which~\eqref{eq:Systeq}  has one solution. Otherwise we obtain
two  distinct pairs $(a,d)$ corresponding to the permutation of the roots of this 
equation. Hence in total, for each of the $3p^2/8 + O(p)$ 
 pairs $(u,v)= (s^2,t^2-2s)$, 
we obtain $2\(p^2/2 +O (p)\) = p^2  + O(p)$ distinct quadruples $(a,b,c,d)$  satisfying~\eqref{eq:Systeq}. Therefore we have
$$
\sharp \mathcal M_2^{\square} (p) = \frac{3}{8} p^4 +O (p^3).
$$
It remains to recall \eqref{1780} and \eqref{1784}  to complete the proof of  ~\eqref{eq: SquareCount}.

\end{document}